\documentclass[11pt,a4paper]{article}

\usepackage{amsmath,amsfonts,amssymb,latexsym,graphics,epsfig,url}
\usepackage{color}
\usepackage{amsthm}
\usepackage[english]{babel}
\usepackage{graphicx}

\newcommand\blfootnote[1]{%
  \begingroup
  \renewcommand\thefootnote{}\footnote{#1}%
  \addtocounter{footnote}{-1}%
  \endgroup
}

\newtheorem{theorem}{Theorem}[section]
\newtheorem{lemma}[theorem]{Lemma}

\DeclareMathOperator{\diag}{diag}

\DeclareMathOperator{\tr}{tr}
\DeclareMathOperator{\spec}{sp}
\def\Re{{\rm Re\,}}

\def\vec0{\mbox{\boldmath $0$}}

\def\A{\mbox{\boldmath $A$}}

\def\CC{\mbox{\boldmath $C$}}

\def\I{\mbox{\boldmath $I$}}

\def\M{\mbox{\boldmath $M$}}

\def\T{\mbox{\boldmath $T$}}

\def\I{\mbox{\boldmath $I$}}

\def\C{\mathbb C}

\def\Re{\mathbb R}

\begin{document}
\title{A new approach to gross error\\
detection for GPS networks \thanks{This research have been partially supported by the {\em Catalan Research Council}, AGAUR, under project 2017SGR1087.}}

\author{C. Dalf\'o$^a$, M. A. Fiol$^b$
\\ \\
{\small $^{a,b}$Departament de Matem\`atiques} \\
{\small  Universitat Polit\`ecnica de Catalunya} \\
{\small $^b$Barcelona Graduate School of Mathematics} \\
{\small Barcelona, Catalonia} \\
{\small {\tt\{cristina.dalfo,miguel.angel.fiol\}@upc.edu}} \\
}
\date{}
\maketitle

\maketitle

\begin{abstract}
We present a new matrix-based approach to detect and correct gross errors in GPS geodetic control networks. The study is carried out by introducing a new matrix, whose entries are powers of a (real or complex) variable, which fully represents the network.
\end{abstract}

\blfootnote{
\begin{minipage}[l]{0.3\textwidth} \includegraphics[trim=10cm 6cm 10cm 5cm,clip,scale=0.15]{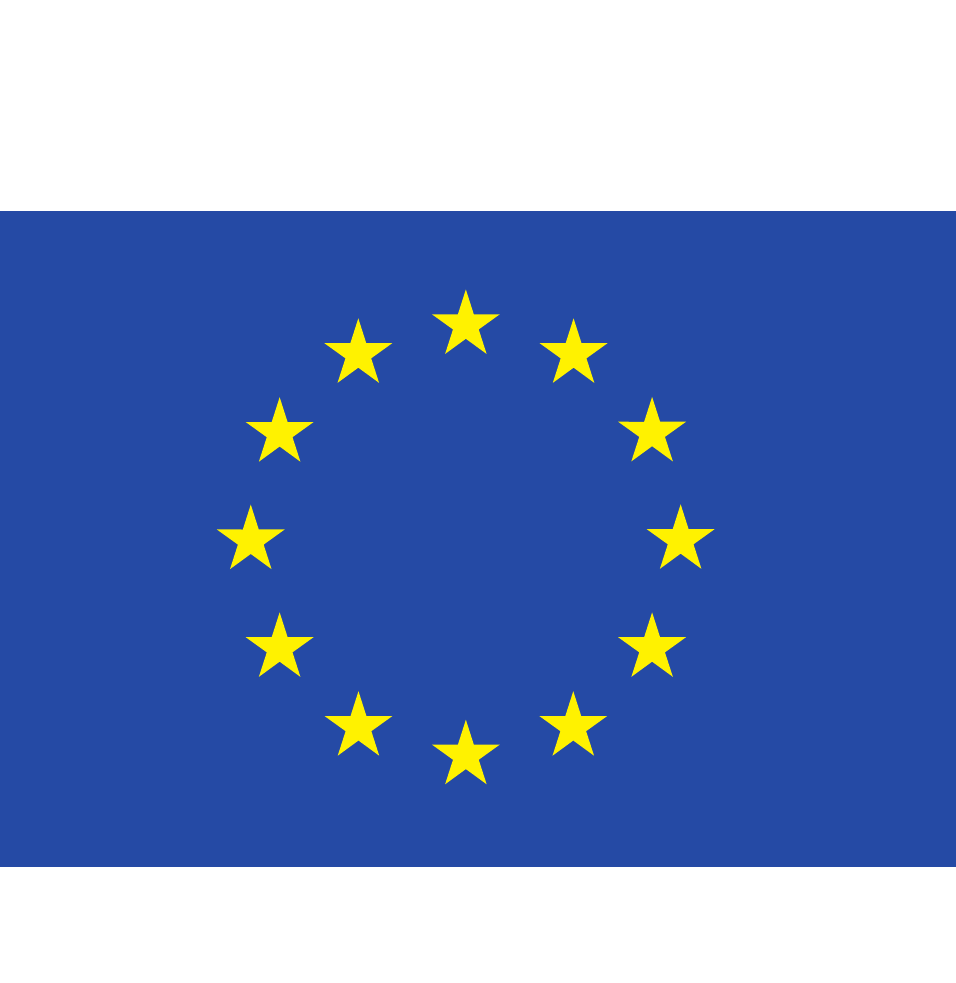} \end{minipage}  \hspace{-2cm} \begin{minipage}[l][1cm]{0.79\textwidth}
   The first author has also received funding from the European Union's Horizon 2020 research and innovation programme under the Marie Sk\l{}odowska-Curie grant agreement No 734922.
  \end{minipage}}



\section{Introduction}

Basically, a GPS network can be seen as a set of nodes, representing spatial points $u\in \Re^3$, together with some information about their exact positions $(u(x_1),u(x_2),u(x_3))$. Thus, this network can be modeled with a directed graph (or digraph) with three isomorphic components $G_i=(V_i,E_i)$, one for each value of the coordinates $x_i$, for $i=1,2,3$, with vertices representing the distinct points. Moreover, there is an arc $uv\in E_i$ from vertex $u$ to vertex $v$ if we measure the three components $\Delta x_i=v(x_i)-u(x_i)$, for $i=1,2,3$, of the GPS vector from $u$ to $v$. Because of the isomorphism between the $G_i$'s, we can concentrate our study to only one component, say $G=(V,E)=G_1$ with arcs giving information about one coordinate, say, $x=x_1$. Then, if $G$ is a (connected) tree, we can extract a unique value for the components $x(u)$ of each vertex $u\in V$. However, in order to improve the accuracy of the positions, we have somehow redundant measures, which implies the presence of (not necessarily directed) cycles in $G$. Then, in this framework, through each cycle the components $\Delta x(u)$ ideally would sum up to zero (satisfying the so-called `loop law'). Because of the error measures, in general this is not the case and such a sum is required not to exceed a maximum tolerated value. Otherwise, the network must be readjusted by ignoring or changing some measures (that is, removing or changing the weights of some arcs).
In order to do this, some methods have been proposed (see Even-Tzur \cite{et01}, and Katambi, Guo, and Kong \cite{kjx02}). They are based on the determination of a fundamental set of independent cycles, by using the spanning trees and incidence matrix associated to the network.

In this paper, we propose a new method of detecting error measures and correcting (or removing) them.
This study is carried out by introducing a new matrix, whose entries are powers of a (real or complex) variable. This matrix fully represents the weighted digraph.

\section{Preliminaries}

First, we recall some basic notation and results. 
A digraph $G=(V,E)$ consists of a (finite) set
$V=V(G)$ of $n=|V|$ vertices, and a set $E=E(G)$ of $m=|E|$ arcs (directed edges) between vertices of $G$.
If $a=uv\in E$ is an arc from $u$ to $v$, then vertex $u$ (and arc $a$) is
{\em adjacent to} vertex $v$, and vertex $v$ (and arc $a$) is {\em adjacent from} $v$. Let $G^+(v)$ and $G^-(v)$ denote, respectively, the set of vertices adjacent from and to vertex $v$. 
A digraph $G$ is
$d$\emph{-regular} if $|G^+(v)|=|G^-(v)|=d$ for all $v\in V$.
Recall also that a digraph $G$ is \emph{strongly connected} if there is a (directed) walk between every
pair of its vertices. If the {\em underlying graph} $UG$, obtained by turning arcs into (undirected) edges, is connected, then the digraph is called \emph{weakly connected}.
The {\em adjacency matrix} $\A=(\A)_{uv}=(a_{uv})$ of $G=(V,E)$ is a $0$-$1$ matrix indexed by the vertices of $V$ and has entries $a_{uv}=1$ if $uv\in E$, and $a_{uv}=0$ otherwise. The spectrum of a (diagonalizable)  matrix $\M$ is denoted by
$$
\spec \M=\{[\lambda_0]^{m_0},[\lambda_1]^{m_1},\ldots,
[\lambda_d]^{m_d}\},
$$
where $\lambda_0>\lambda_1>\cdots >\lambda_d$ and the superscripts stand for multiplicities, or simply by a vector $(\theta_1,\theta_2,\ldots,\theta_n)$ whose entries are the $n$ eigenvalues (with possible repetitions).
For other basic concepts and definitions about digraphs and algebraic graph theory, see for instance 
Biggs \cite{b93}
Bang-Jensen and Gutin~\cite{BG07}
or Diestel \cite{d10}.

We also  need to recall the following known result (see for instance  Gould \cite{go99}.).

\begin{lemma}
\label{gould}
If the power sums $s_k=\sum_{i=1}^n z_i^k$ of some complex numbers $z_1,\ldots,z_n$ are known for every $k=1,\ldots,n$, then such numbers are
the roots of the monic polynomial
$$
p(z)=\frac{1}{n!}\det \CC(z),
$$
where $\CC(z)$ is the following matrix of dimension $n+1$:
$$
\CC(z)=
\left(
\begin{array}{cccccccc}
z^n & z^{n-1} & z^{n-2} & z^{n-3} & \cdots & z^2 & z & 1 \\
s_1 & 1       &   0     &    0    &  \cdots & 0 & 0 & 0 \\
s_2 & s_1     &   2     &    0    &  \cdots & 0 & 0 & 0 \\
s_3 & s_2     &   s_1   &    3    &  \cdots & 0 & 0 & 0 \\
\vdots & \vdots     &   \vdots   &    \vdots    &  \ddots & \vdots & \vdots & \vdots \\
s_{n-1} & s_{n-2}     &   s_{n-3}   &    s_{n-4}    &  \cdots & s_1 & n-1 & 0 \\
s_n     & s_{n-1} & s_{n-2}     &   s_{n-3}   &     \cdots & s_2 & s_1 & n \\
\end{array}
\right).
$$
\end{lemma}

\section{A new matrix representation}

Let $G=(V,E)$ be a digraph with a weight function $w:E\rightarrow \Re$, which, to every arc $uv$, assigns the measured value of a coordinate, say $x$, of the vector from $u$ to $v$. Notice that there is no loss of generality if we assume that all weights are positive, since we can reverse the direction of some arcs to go from $w(uv)=-\alpha$  to $w(vu)=\alpha$.
Then, its associated {\em polynomial matrix} $\A(z)$ is a square matrix indexed by the vertices of $G$ and whose entries are powers of (a real or complex) variable $z\neq 0$, as follows:
$$
(\A(z))_{uv}=\left\{
\begin{array}{ll}
z^{w(uv)}, & \mbox{if $uv\in E$,}\\
z^{-w(uv)}, & \mbox{if $vu\in E$,}\\
0, & \mbox{otherwise.}
\end{array}\right.
$$
In particular, note that $\A(1)$ coincides with the standard adjacency matrix of the underlying graph $UG$. Thus, it is well known that the $uv$-entry of the power $(\A(1))^r$ is the number of walks of length $r$ in $UG$.

Before giving the corresponding result for an indeterminate $z$, we introduce some additional notation.
For a given arc $uv\in E$, let $w^*(uv)=w(uv)$ and $w^*(vu)=-w(uv)$.
Given two vertices $u,v$ and an integer $r\ge 0$, let ${\cal P}_r(u,v)$ be the set of $r$-walks from $u$ to $v$ in $UG$. For any $p\in {\cal P}_r(u,v)$, say $p=u_1(=u),u_2,\ldots,u_r(=v)$, let $w^*(p)=1$ if $r=1$, and
$$
w^*(p)=\sum_{i=1}^{r-1} w^*(u_iu_{i+1})\quad\mbox{for $r>1$}.
$$
For any pair of vertices $u$ and $v$, the powers of the polynomial matrix gives information about the values of $w^*(p)$, as shown in the following basic lemma.

\begin{lemma}
\label{basic-lemma}
For each $r\ge 0$, the $uv$-entry of the power $\A(z)^r$ is
\begin{equation}
(\A(z)^r)_{uv}=\sum_{p\in {\cal P}_r(u,v)} z^{w^*(p)}.
\end{equation}
\end{lemma}
\begin{proof}
The result is trivial for $r=0$ (since $\A(z)^0=\I$), and it follows from the definition of $\A(z)$ for $r=1$. Now suppose that the result holds for $r=R$.
Then,
\begin{align*}
(\A(z)^{R+1})_{uv} &=\sum_{w\in V}(\A(z)^{R})_{uw}a(z)_{wv}\\
 &=\sum_{wv\in E(UG)}\sum_{p\in {\cal P}_R(u,w)}z^{w^*(p)}z^{w^*(wv)}\\
 &=\sum_{p\in {\cal P}_R(u,w)}\sum_{wv\in E(UG)}z^{w^*(p)+w^*(wv)}
 =\sum_{q\in {\cal P}_{R+1}(u,v)}z^{w^*(q)}.
\end{align*}
The general result follows by induction.
\end{proof}

\section{Detecting and correcting}
In this section, we propose a new method of detecting and correcting gross errors in GPS networks. We first explain how to detect the presence of (significant) errors.

\subsection{Checking the loop law}
As explained in the Introduction, error detection in GPS networks is based on checking the loop law.
With this aim, we are now ready to prove the main result.

\begin{theorem}
\label{main-theo}
Let $G$ be a weighted digraph with polynomial matrix $\A(z)$. Then, $G$ satisfies the loop law if and only if any of the following conditions hold:
\begin{itemize}
\item[$(a)$]
$\diag \A(z)^r=\diag \A(1)^r$
for any $z\neq 0$ and $r=0,1,2,\ldots$.
\item[$(b)$]
$\|\diag \A(z)^r -\diag \A(1)^r \|=0$ for any $z\neq 0,1$ and $r=0,1,2,\ldots$.
\item[$(c)$]
The matrix $\A(z)$ is diagonalizable, with real constant eigenvalues, say $\theta_1,\theta_2,\ldots,\theta_n$, and
$\spec \A(z) =\spec \A(1)$ for any $z\in\C $.
\end{itemize}
\end{theorem}

\begin{proof}
$(a)$ Notice that $G$ satisfies the loop law if and only if $w^*(p)=0$ for any $p\in {\cal P}_r(u,u)$, $u\in V$, and $r\ge 0$. Then,  from Lemma \ref{basic-lemma}, this is equivalent to
$$
(\A(z)^r)_{uu}=\sum_{p\in {\cal P}_r(u,u)} 1= (\A(1)^r)_{uu}
$$
for any $z\neq 0,1$ and $r=0,1,2,\ldots$.
Condition $(b)$ is just a reformulation of $(a)$.
Concerning $(c)$, notice that $\A$ is a symmetric matrix with real eigenvalues, say $\theta_1,\theta_2,\ldots,\theta_n$. Then, from $(a)$,
$$
\tr \A(z)^r=\tr \A(1)^r=\sum_{i=1}^n \theta_i^r
$$
for any $z\in\C$ and $r=1,\ldots,n$. But, by Lemma \ref{gould}, this holds if and only if
$\spec \A(z) =\spec \A(1)$ for any $z\in\C $, as claimed.
\end{proof}

\begin{figure}[t]
\begin{center}
\includegraphics[width=12cm]{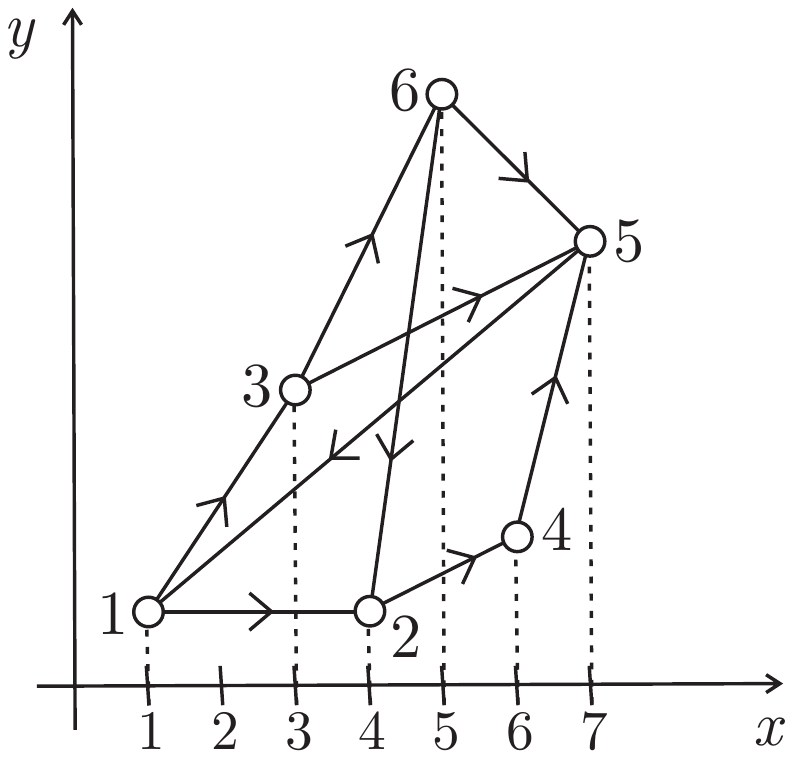}
\end{center}
\vskip-13cm
\caption{A simple GPS network with 6 nodes.}
  \label{fig1}
\end{figure}

\subsection{Correcting gross errors}

First, notice that the norm in condition of Theorem \ref{main-theo}$(b)$ (or distance between the corresponding diagonal vectors) gives us a measure of the correctness of the measures, so that we can require not to exceed a maximum tolerated value.

If that is not the case, we have two options. The most simple one is to remove the suspected error measures, that is, the corresponding entries of the polynomial matrix $\A(z)$, and check again whether the norm in such a condition $(b)$ does not exceed the allowed  maximum.

The second option is to try to correct the alleged errors as much as possible. To this end, we can proceed as follows:
\begin{enumerate}
\item
Assuming that the wrong values are at the arcs $u_iv_i$, for $i=1,\ldots,s$, consider the new polynomial matrix $\A'(z)$ obtained from $\A(z)$ by changing each weight $w^*(u_iv_i)$ to
$w^*(u_iv_i)+x_i$, where $x_i$ represents the error in the measure for $i=1,\ldots,s$.
\item
For the fist value of $r$ that the loop law is not satisfied according to Theorem \ref{main-theo}, and for some value of $z\neq 1$, say $z=2$, find the values of $e_i$
that minimize any of the following functions:
\begin{itemize}
\item
$e_a(x_1,\ldots,x_s)=\diag \A(2)^r-\diag \A(1)^r$.
\item
$e_b(x_1,\ldots,x_s)=\|\diag \A(2)^r-\diag \A(1)^r\|$.
\item
$e_c(x_1,\ldots,x_s)=\spec \A(2)^r-\spec \A(1)^r$.
\end{itemize}
\end{enumerate}

\section{Examples}
Let us see how our method applies by following two examples.
The first one only intends illustrate the method in a very simple situation, while the second deals with experimental data drawn from \cite{et01}.

\subsection{A simplified example}
For a simple example, look at the digraph of Figure \ref{fig1} and suppose that the weights of the arcs are the ones indicated by the exponents in the entries of the following polynomial matrix (which correspond to the abscissas of the points):
\begin{equation}
\label{matrix-1}
\A(z)=
\left(
\begin{array}{cccccc}
0& z^3& z^2& 0& z^6& 0 \\
z^{-3}& 0& 0& z^2& 0& z \\
z^{-2}& 0& 0& 0& z^4& z^2 \\
0& z^{-2}& 0& 0& z& 0\\
z^{-6}& 0& z^{-4}& z^{-1}& 0& z^{-2}\\
0  & z^{-1} & z^{-2}& 0& z^2& 0
\end{array}
\right).
\end{equation}
Then, the diagonal vectors of the testing matrices, say,  $\T^{(2)}=\A(2)^r-\A(1)^r$ have zero norm for any $r\ge 0$, showing that there are no errors in the measures (weights of the arcs).

Alternatively, we reach the same conclusion by computing the spectrum
$$
\spec \A(z)=\{ [3.092]^1,[0.702]^1, [0]^2,[-1.285]^1,[-2.508]^1 \},
$$
which is independent of $z$, as expected since $\tr(\A(z)^r)=\tr(\A(1)^r)$ for every $r=0,1,2,\ldots$

Now, consider a new set of arc weights given by the matrix
\begin{equation}
\A(z)=
\left(
\begin{array}{cccccc}
0& z^3& z^2& 0& z^6& 0 \\
z^{-3}& 0& 0& z^2& 0& z^2 \\
z^{-2}& 0& 0& 0& z^4& z^2 \\
0& z^{-2}& 0& 0& z& 0\\
z^{-6}& 0& z^{-4}& z^{-1}& 0& z^{-2}\\
0& z^{-2}& z^{-2}& 0& z^2& 0
\end{array}
\right).
\end{equation}
Then, the norms of the diagonal vectors of the testing matrices $\T^{(r)}=\A(2)^r-\A(1)^r$ for $r=1,\ldots,6$, are
$$
 0,\ 0,\ 0,\ 6,\ 15/2,\ 78,
$$
indicating that there is some error in the measures.
To detect it, we observe that the diagonal of the matrix $\T^{(4)}$ (since $4$ is the first value of $r$ for which the diagonal of $\T^{(r)}$ is not null). That is,
$$
\textstyle
\diag \T^{(4)}=\left(1,\frac{3}{2},\frac{1}{2},\frac{1}{2},1,\frac{3}{2}\right),
$$
whence, by looking at the largest entries, we conclude that, most probably, the error is the measure between the points 2 and 6.
In fact, this phenomenon is made more apparent by looking at the following limit:
$$
\lim_{z\rightarrow \infty}\left(\frac{1}{z}(\A(z)^4-\A(1)^4)\right)=(2,3,1,1,2,3).
$$

If the arc between points $2$ and  $6$ has weight $2+x$, then we get an ``error" as a function of $x$. Namely,
$$
e(x)=\tr\A(2)^4-\tr\A(1)^4=24\cdot 2^x+6\cdot 2^{-x}-24,
$$
which has minimum value zero when $x=-1$, so indicating that the correct weight must be $1$, as shown in  \eqref{matrix-1}.

\begin{figure}[h]
\begin{center}
\includegraphics[width=6cm]{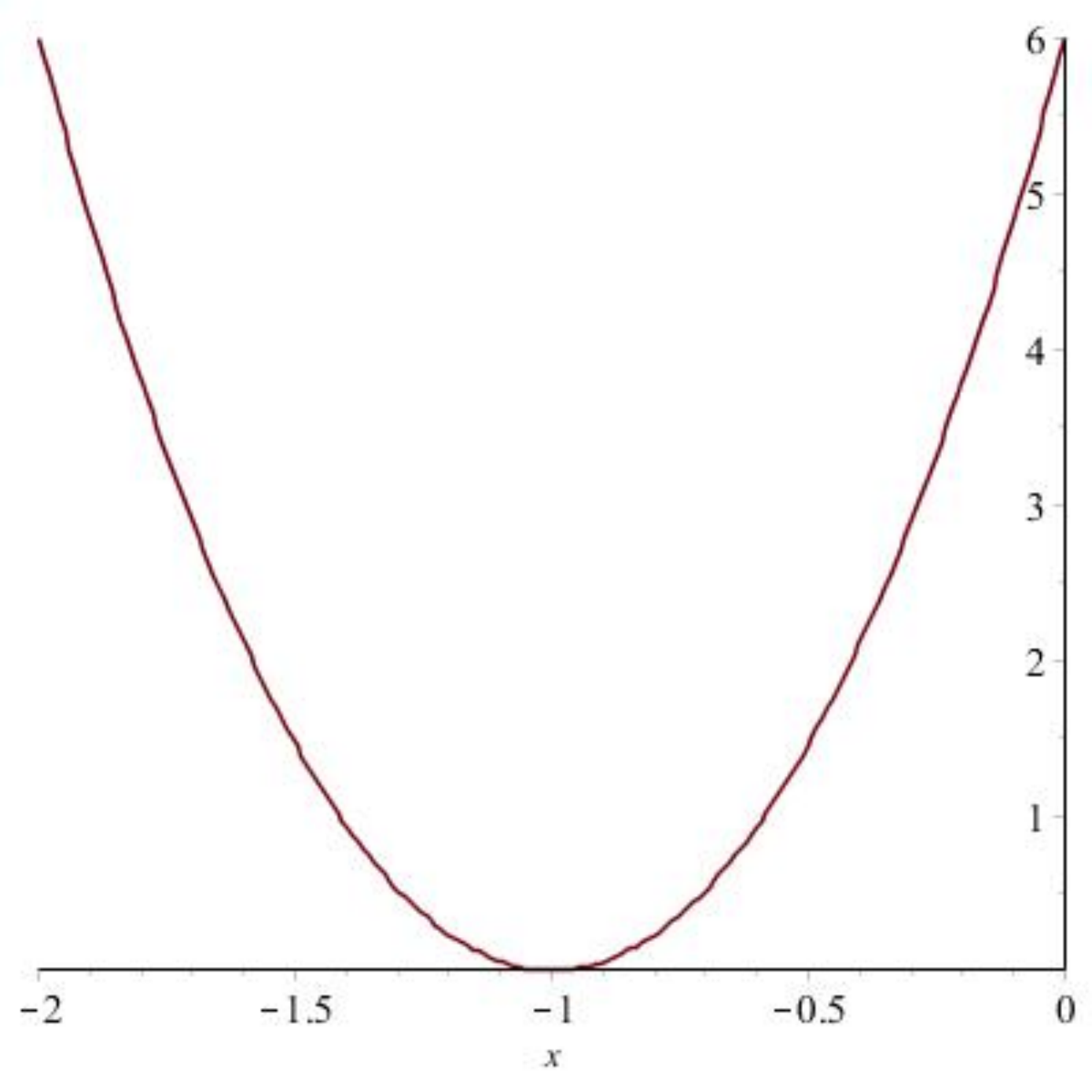}
\end{center}
\vskip-.6cm
\caption{The value of $e(x)$ as a function of $x$.}
  \label{fig3}
\end{figure}

\subsection{An example from experimental data}
This example is taken from the experimental data given by Even-Tzur \cite{et01} (see Figure \ref{fig2}):

\begin{figure}[t]
\begin{center}
\includegraphics[width=12cm]{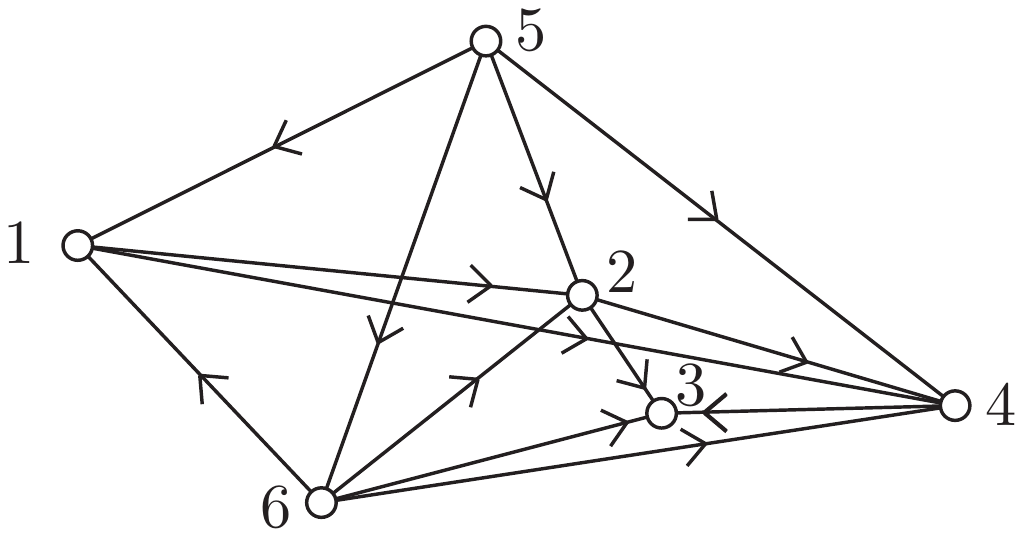}
\end{center}
\vskip-14.2cm
\caption{An experimentally obtained GPS network with 6 nodes.}
  \label{fig2}
\end{figure}

For the $x$-components, the polynomial matrix is 
$$
\A(z)=
\left(
\begin{array}{cccccc}
0& z^{-1718.388} & 0& z^{-5472.823} & z^{-2265.013} & z^{1538.763} \\
z^{1718.388} & 0 & z^{724.719} & z^{-3754.454} & z^{-546.654} & z^{3257.146}\\
0& z^{-724.719} & 0 & z^{-4479.188} & 0 & z^{2532.446} \\
z^{5472.823} & z^{3754.454} & z^{4479.188} & 0& z^{3207.794} & z^{7011.591}\\
z^{2265.013} & z^{546.654} & 0& z^{-3207.794} & 0& z^{3803.805} \\
z^{-1538.763}& z^{-3257.146} & z^{-2532.446} & z^{-7011.591} & z^{-3803.805} & 0
\end{array}
\right)
$$
The norms of the diagonal vectors of the testing matrices $\T^{(r)}=\A(2)^r-\A(1)^r$ for $r=1,\ldots,6$, are
$$
 0,\ 0,\ 0.003,\ 0.010,\ 0.075,\ 0.362,
$$
indicating that there is some error in the measures.
To check how important it is, we look at the diagonal of the matrix $\T^{(3)}=\A(2)^3-\A(1)^3$, which is,
$$
(0.00113,\ 0.00094,\ 0.00117,\ 0.00139,\ 0.00099,\ 0.00157)
$$
(and similarly for other values of $r>3$)
showing that, since all the entries are `small enough', say, of the order of $10^{-3}$, there are no gross errors for this coordinate in the GPS measurements.

Now, suppose that, as in \cite{et01}, that an error occurred in the vectors from point $1$ to $4$, and from point $5$ to $4$, with respective $x$-measurements $-5.472.959$ (instead of the previous $-5472.823$, and $-2307.933$ (instead of the previous $-2307.794$).
Then, the norms of the diagonals of the matrices $\T^{(r)}=\A(2)^r-\A(1)^r$ for $r=1,\ldots,6$, namely
$$
 0,\ 0,\ 0.050,\ 0.157,\ 1.171,\ 5.505,
$$
indicate that the errors are larger than before.
Indeed, the diagonal vector of the matrix $\T^{(3)}=\A(2)^3-\A(1)^3$ is now,
$$
(0.01573,\ 0.01583,\ 0.00117,\ 0.03485,\ 0.01980,\ 0.02017),
$$
showing errors of the order of $10^{-2}$ which, with the above criterium, would be considered as ``gross errors". Moreover, the largest entry of the above vector corresponds to point $4$, showing possible errors in the corresponding arcs.

If the arc between points $4$ and  $5$ has weight $3207.933+x$, and the arc between points $1$ and $4$ has weight $5472.959+y$. Then we get an ``error" as a function of $x,y$. Namely,
\begin{align*}
e(x,y) & =\tr\A(2)^4-\tr\A(1)^4\\
 &=8.09\cdot 2^{x-y}+7.91\cdot 2^{y-x}+26.09\cdot 2^x+22.07\cdot 2^{-x}+26.16\cdot 2^y+22.03\cdot 2^{-y}-112,
\end{align*}
which has minimum value $0.02$ at $(x,y) = (-0.124, -0.120)$ (see Figure \ref{fig4}), so indicating that the corrected values are, respectively, $3207.933-0.124=3207.809$ and $5472.959-0.120=5472.839$. These values are very closed to those not implying gross errors, namely, $3207.794$ and $5472.823$.

\begin{figure}[h]
\begin{center}
\includegraphics[width=8cm]{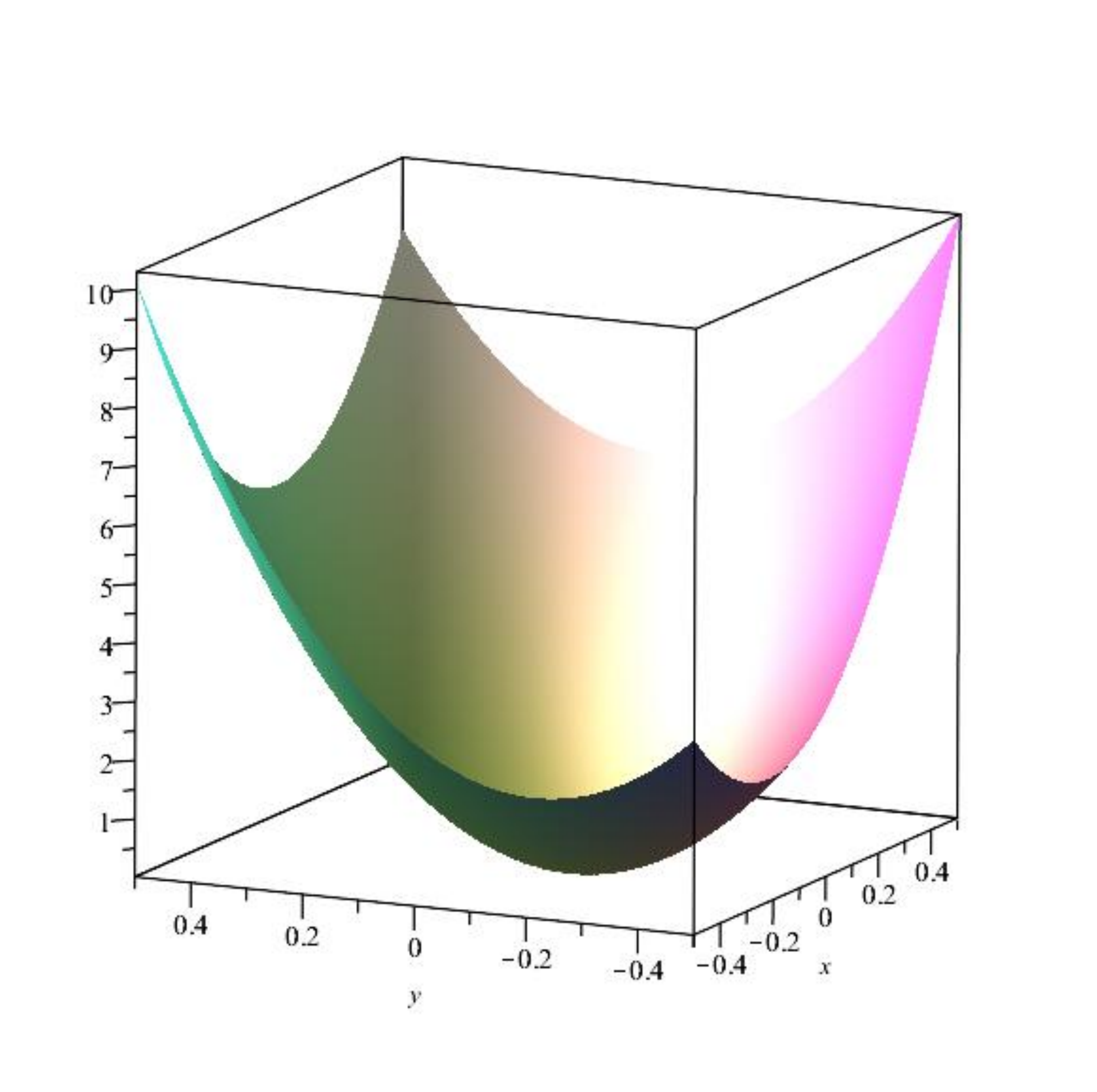}
\end{center}
\vskip-.8cm
\caption{The value of $e(x,y)$ as a function of $x,y$.}
  \label{fig4}
\end{figure}


\end{document}